\documentclass{amsart}
      \makeatletter
     \def\section{\@startsection{section}{1}%
     \z@{.7\linespacing\@plus\linespacing}{.5\linespacing}%
     {\bfseries%\normalfont\scshape
     \centering
     }}
     \def\@secnumfont{\bfseries}
     \makeatother
\newtheorem{theorem}{Theorem}[section]
\newtheorem{lemma}[theorem]{Lemma}
\newtheorem{corollary}[theorem]{Corollary}

 \usepackage{a4wide}
\theoremstyle{definition}

\theoremstyle{remark}
\newtheorem{remark}[theorem]{Remark}

\numberwithin{equation}{section}

\def \a{{\alpha}}

\def \d{{\delta}}
\def \e{{\varepsilon}}

\def \g{{\gamma}}

\def \l{{\lambda}}

\def \O{{\Omega}}

\def \t{{\vartheta}}

\def \m{{\mu}}
\def \s{{\sigma}}

\def \P{{\bf P}}

\def \qq{{\qquad}}
\def \R{{\bf R}}

 \def \dd{{\rm d}}

\def \noi{{\noindent}}

\def\E{{\mathbb E \,}}

\def\P{{\mathbb P}}

\def\R{{\mathbb R}}

\def\C{{\mathbb C}}  
  
at  8 pt
\scrollmode

   at 10,4  pt at  8,2 pt
  \scrollmode
 at 10,2  pt

  \title[Cauchy Means of Dirichlet polynomials]{Cauchy Means of Dirichlet polynomials}
\begin{document}  \author{\rm       Michel J.\,G. Weber}
 \address{IRMA, Universit\'e Louis-Pasteur et C.N.R.S.,   7  rue Ren\'e Descartes, 67084 Strasbourg Cedex, France. 
  \noi E-mail:    {\tt  michel.weber@math.unistra.fr }}

  %%%%%%%%%%%%%%%%%%%%%%%%%%%%%%%%%%%%%%%%%%%%%%%%%%%%%%%%%%%%%%%%%%%%%%%%%%%%%%%%%%%%%%%%%%%%%%

  %  \baselineskip= 12 pt

  \begin{abstract} We study Cauchy means of
Dirichlet polynomials   
$$\int_\R  \Big|\sum_{n=1}^N 
\frac{1}{ n^{\s+ ist}}   \Big|^{2q}  
 \frac{\dd t}{\pi( t^2+1)}.$$ 
 These integrals were investigated when $q=1,\s= 1, s=1/2 $ by Wilf, using integral operator theory and Widom's eigenvalue estimates.
 We show the optimality of some upper bounds obtained by Wilf. We also obtain new
estimates for the case $q\ge 1$, $\s\ge 0$ and $s>0$. We complete  Wilf's approach by relating it with other approaches  (having notably  connection with Brownian motion), allowing simple proofs, and also  prove new  results. 
% We propose a   different and more  direct 
%approach, which has a natural  connection with Brownian motion.
\end{abstract}
 
 \maketitle 
 
%%%%%%%%%%%%%%%%%%%%%%%%%%%%%%%%%%%%%%%%%%%%%%%%%%%%%%%%%%%%%%%%%%%%%%%%%%%%%%%%%%%%%%%%%%%%%%%%%
%%%%%%%%%%%%%%%%%%%%%%%%%%%%%%%%%%%%%%%%%%%%%%%%%%%%%%%%%%%%%%%%%%%%%%%%%%%%%%%%%%%%%%%%%%%%%%%%%

\section{\bf Introduction and Main Results.} 
 In a quite   inspiring  paper \cite{Wi}, Wilf has considered   integral  operators  associated with homogeneous, nonnegative  kernels 
$K(x,y)$ and applied his results to  Dirichlet series. 
Consider for instance the kernel $K(x,y)=\max(x,y)^{-1}$. It
has   Mellin  transform
$$\mathcal F(s)=\int_0^\infty t^{-s}K(t,1) \dd t= \frac{1}{s}+\frac{1}{1-s}, $$
$s=\s+it$, which is invertible on the critical line. As further $K(x,y)$ is symmetric and decreasing, it is well-known   in this case  that  the spectral theory of
$K(x,y)$   depends on the  behavior of the Mellin  transform of $K(t,1)$ along the critical line. 

\vskip 2 pt If
$x_1, \ldots , x_N$ are complex numbers, then (\cite{Wi}, Theorem 3)
\begin{equation}\label{Wilf}\sum_{n,m=1}^N \bar x_nK(n,m)x_m=\frac{1}{2\pi}\int_\R \mathcal F(\frac{1}{2}+ it)\Big|\sum_{n=1}^N 
\frac{x_n}{ n^{\frac{1}{2}+ it}}   \Big|^2 \dd t\le 
\mathcal F(\frac{1}{2} 
 )\sum_{j=1}^N | x_n|^2.
\end{equation}
 The last inequality follows from Widom's  eigenvalue estimate   (\cite{Wi},  Theorem  2). Wilf has shown that (\ref{Wilf}) holds for the class $\mathcal
H$ of kernels $K$ such that
$K(x,y)\ge 0$ for $x,y$ nonnegative, and is further symmetric,  decreasing and homogeneous of degree $-1$: for every $\a>0$ we
 have 
 \begin{equation}\label{ho-1} K(\a x, \a y)= \a^{-1}K(x,y) \qq \quad \forall x>0,\forall y>0.
 \end{equation}
 
In the case considered, (\ref{Wilf}) implies that
 \begin{equation}\label{Wilf1}  \int_\R   \Big|\sum_{n=1}^N 
\frac{x_n}{ n^{\frac{1}{2}+ it}}   \Big|^2 \frac{\dd t  }{\frac{1}{4}+  t^2} \le 
8\pi\sum_{n=1}^N | x_n|^2.
\end{equation}
Taking $x_n= n^{-1/2}$ yields in particular the following nice bound (\cite{Wi}, (17))
 \begin{equation}\label{Wilf2}  \int_\R   \Big|\sum_{n=1}^N 
\frac{1}{ n^{1+ it}}   \Big|^2 \frac{\dd t  }{\frac{1}{4}+  t^2} \le 
8\pi\sum_{n=1}^N \frac{1  }{n}\le C \log N.
\end{equation}

That  inequality   is in turn  two-sided and this can
be showed     without appealing to Mellin
transform nor Widom's eigenvalue  estimate. 
 The purpose of this Note is to first  relate  Wilf's approach with other approaches  allowing simple proofs, and next,  to develop more some parts and  prove new  results.  
 The above integrals   are  Cauchy means on the real line of Dirichlet polynomials,  
  and admit an exact formulation. This is in contrast with   usual mean-value  of Dirichlet polynomials, with respect to   measures
$\chi_{[0,T]}(t)\dd t/T$,      where  an error term  
always occurs due to the fact that 
\begin{equation}\label{omn}\int_{0}^T \Big(\frac{m}{n}\Big)^{it} \dd t=\begin{cases} \, T\qq &{\rm if} \ m=n\cr 
 \,\mathcal O_{m,n}(1)\qq &{\rm otherwise}.
\end{cases}\end{equation}   

 Both means are in turn strongly related. Cauchy means of Dirichlet polynomials are part of the theory of Dirichlet polynomials for various weights and it is expected that their study will give new insight into properties of  general Dirichlet polynomials. We refer for instance to the recent works of Lubinsky \cite{L1, L2}, 
 [which we discovered while this work was much advanced]. 
 \vskip 2 pt 
 As the weight  functions in turn represent a sampling of the parameter $t$, the properties of the weighted Dirichlet approximating polynomials can be used to study  the behavior of the Riemann zeta function $\zeta(\s + it)$ along the critical line $\s=1/2$. A (rather) elaborated application  of this, in the case of the Cauchy density, can be found in  Lifshits and Weber \cite{LW}.
 \vskip 2 pt 
We begin with giving  proofs of (\ref{Wilf1}), (\ref{Wilf2}) without appealing to spectral theory (Widom's eigenvalue  estimate). 

%%%%%%%%%%%%%%%%%
%%%%%%%%%%%%%%%%%
\subsection{Proof of  (\ref{Wilf2}) using Cauchy means} We start with an elementary lemma.   
 \begin{lemma}\label{debut} Let  $s\in \R_+$ and $x_1,\ldots, x_N$,
$y_1,\ldots, y_N$ be complex numbers. We have
\begin{eqnarray*} \int_\R \Big|\sum_{\nu=1}^My_\nu \nu^{ i s t}\Big|^2
\Big|\sum_{n=1}^Nx_n n^{ i s t}\Big|^2\frac{\dd t}{\pi( t^2+1)}&=&  \sum_{\m,\nu=1}^M\sum_{m,n=1}^N \overline{y _\m x _m}\, y_\nu  x_n \Big(\frac{n\nu 
 \wedge m\mu}{n\nu  \vee m\mu }\Big)^s
\end{eqnarray*}
Moreover, 
\begin{eqnarray*} \lim_{s\to \infty}\int_\R \Big|\sum_{\nu=1}^My_\nu \nu^{ i s t}\Big|^2
\Big|\sum_{n=1}^Nx_n n^{ i s t}\Big|^2\frac{\dd t}{\pi( t^2+1)}&=&  \sum_{{1\le \m,\nu\le M\atop 1\le m,n\le N}\atop n\nu=m\m} \overline{y _\m x _m}\, y_\nu  x_n .
\end{eqnarray*}
 \end{lemma}
\begin{remark} The last assertion implies   that
\begin{eqnarray*} \lim_{s\to \infty}\int_\R\Big|\sum_{n=1}^Nx_n n^{ i s t}\Big|^2\frac{\dd t}{\pi( t^2+1)}&=&  \sum_{n=1}^N   |x_n|^2\cr  \lim_{s\to \infty}\int_\R \Big|\sum_{n=1}^Nx_n n^{ i s t}\Big|^4\frac{\dd t}{\pi( t^2+1)}&=&  \sum_{{1\le \m,\nu\le N\atop 1\le m,n\le N}\atop n\nu=m\m} \overline{x _\m x _m}\, x_\nu  x_n .
\end{eqnarray*}Taking $x_n=n^{-\s}$ yields,
\begin{eqnarray*} \lim_{s\to \infty}\int_\R\Big|\sum_{n=1}^N\frac{1}{n^{ \s +i s t}}\Big|^2\frac{\dd t}{\pi( t^2+1)}&=&  \sum_{n=1}^N   \frac{1}{n^{2 \s }}\cr  \lim_{s\to \infty}\int_\R \Big|\sum_{n=1}^N\frac{1}{n^{ \s +i s t}}\Big|^4\frac{\dd t}{\pi( t^2+1)}&=&  \sum_{{1\le \m,\nu\le N\atop 1\le m,n\le N}\atop n\nu=m\m}   \frac{1}{(\m m\nu n)^{\s}}.
\end{eqnarray*}
And in particular, by using Ayyad, Cochrane and Zheng estimate \cite{ACZ}, Theorem 3,
\begin{eqnarray*}  \lim_{s\to \infty}\int_\R \Big|\sum_{n=1}^N\frac{1}{n^{i s t}}\Big|^4\frac{\dd t}{\pi( t^2+1)}&=&  \#\Big\{1\le \m,\nu\le N, 1\le m,n\le N : n\nu=m\m\Big\}\cr &=& \frac{12}{\pi^2}N^2\log N + CN^2 + \mathcal O \big( N^{19/13}\log^{7/13} N\big),
\end{eqnarray*}where $C= \frac{2}{\pi^2}(12\g -(\frac{36}{\pi^2}\zeta'(2)-3)-2$, $\g$ is Euler's constant and  $\zeta'(2)= \sum_{n=1}^\infty  \frac{\log n}{n^2}. $
 \end{remark}
 \begin{proof}
 From the relation    
  $ e^{-|\t|}= \int_\R  e^{i\t t}\frac{\dd t}{\pi( t^2+1)}$, 
it follows that 
\begin{eqnarray} \label{cor1}\Big(\frac{n }{m}\Big)^s=\int_\R 
 \frac{1}{  n^{ ist}m^{-ist}}\frac{\dd t}{\pi( t^2+1)} \qq \qquad (m\ge n). 
\end{eqnarray}  
 Thus
\begin{eqnarray*} \int_\R \Big|\sum_{\nu=1}^My_\nu \nu^{ i s t}\Big|^2
 \frac{1}{  n^{ i st}m^{-is t}}\frac{\dd t}{\pi( t^2+1)}&=&\sum_{\m,\nu=1}^M \overline{y}_\m y_\nu\int_\R  \frac{1}{ (n \nu)^{
ist}(m\mu)^{-ist}}
  \frac{\dd t}{\pi( t^2+1)}\cr &=&  \sum_{\m,\nu=1}^M \overline{y}_\m y_\nu\Big(\frac{n\nu  \wedge m\mu}{n\nu  \vee m\mu }\Big)^s .
\end{eqnarray*} 
 Consequently
\begin{eqnarray*} \int_\R \Big|\sum_{\nu=1}^My_\nu \nu^{ i s t}\Big|^2
\Big|\sum_{n=1}^Nx_n n^{ i s t}\Big|^2\frac{\dd t}{\pi( t^2+1)}&=&\sum_{\m,\nu=1}^M\sum_{m,n=1}^N \overline{y}_\m y_\nu\int_\R  \frac{1}{ (n \nu)^{ ist}(m\mu)^{-ist}}
  \frac{\dd t}{\pi( t^2+1)}\cr &=&  \sum_{\m,\nu=1}^M\sum_{m,n=1}^N \overline{y _\m x _m}\, y_\nu  x_n \Big(\frac{n\nu 
 \wedge m\mu}{n\nu  \vee m\mu }\Big)^s.
\end{eqnarray*}
  The second assertion follows easily. Let 
  $$\d=   \max_{{1\le \m,\nu\le M\atop 1\le m,n\le N}\atop n\nu\neq m\m}\Big(\frac{n\nu  \wedge m\mu}{n\nu  \vee m\mu }\Big)^s.$$
Then $0< \d<1$. And the conclusion follows from
\begin{eqnarray*} \int_\R \Big|\sum_{\nu=1}^My_\nu \nu^{ i s t}\Big|^2
\Big|\sum_{n=1}^Nx_n n^{ i s t}\Big|^2\frac{\dd t}{\pi( t^2+1)} &=&   \sum_{{1\le \m,\nu\le M\atop 1\le m,n\le N}\atop n\nu=m\m} \overline{y _\m x _m}\, y_\nu  x_n+ (MN)^2 \mathcal O(\d^s).
\end{eqnarray*}
  \end{proof}

 To recover (\ref{Wilf2}) and also to prove the corresponding lower bound, take $x_n= n^{-1}$, $M=1=y_1$ and $s=1/2$. We get
\begin{eqnarray*}\label{24}  \frac{1}{2\pi } \int_\R  \Big|\sum_{n=1}^N 
\frac{1}{ n^{1+ i \theta  }}   \Big|^2  
 \frac{\dd \theta}{ \frac{1}{4}+ \theta^2 } 
 &=&  \sum_{m,n=1}^N \frac{ 1 }{(m\wedge n)^{1/2} (m\vee n)^{3/2}}
 \cr&=& \sum_{ n=1}^N \frac{ 1 }{n^2}  +2 \sum_{ n=1}^N  \frac{ 1 }{n^{1/2} } \sum_{m=n+1}^N \frac{ 1 }{ m^{3/2}}
\cr&\le &  C
\Big(1+  
\sum_{ n=1}^N  \frac{ 1 }{n  } \Big)\le C\log N ,\end{eqnarray*}
which is (\ref{Wilf2}). And obviously, 
\begin{eqnarray*}\label{24}  \frac{1}{2\pi } \int_\R  \Big|\sum_{n=1}^N 
\frac{1}{ n^{1+ i \theta  }}   \Big|^2  
 \frac{\dd \theta}{ \frac{1}{4}+ \theta^2 } 
 &\ge &  2 \sum_{ n=1}^{N/2}  \frac{ 1 }{n^{1/2} } \sum_{m=n+1}^N \frac{ 1 }{ m^{3/2}}
\ge C\sum_{ n=1}^N  \frac{ 1 }{n  }  \ge C\log N .\end{eqnarray*}
 
%%%%%%%%%%%%%%%%%%%%%%%
%%%%%%%%%%%%%%%%%%%%%%%
\subsection{Proof of  (\ref{Wilf1}) using Brownian motion} 
%We use the fact that the kernel $K(s,t)$ is  directly   related to Brownian %motion. 
Let  $W=\{W(t), t\ge 0\}$ be  standard one-dimensional Brownian motion
issued from $0$ at time $t=0$ and with underlying probability space $(\O, \mathcal A, \P)$.  Then
\begin{equation} \label{kb}  K(s,t)= \frac{(s\wedge t)}{ s t}=  \E \Big( \frac{W(s)}{
s}\frac{W(t)}{  t}\Big)\qq {\rm and}\qq \E \Big( \frac{W(s)}{
\sqrt s}\frac{W(t)}{ \sqrt t}\Big)=\Big(\frac{s\wedge t}{s\vee t} \Big)^{1/2}.
\end{equation} 
%[Notice that similarly for the Hilbert kernel $\E ( \frac{W(s)}{
 %s}\frac{W(s+t)}{s+ t})=\frac{1}{s+t}$].  
  This allows to interpret these integrals  as Brownian sums, and by using the independence of the increments of $W$, to find another convenient
reformulation.  
\begin{lemma}\label{w} For any real $s\ge 0$,
  \begin{eqnarray*} \int_\R\Big|\sum_{n=1}^Nx_n n^{ i s t}\Big|^2
 \frac{\dd t}{\pi( t^2+1)}&=& \E\Big|\sum_{n=1}^Nx_n\frac{W(n^{2s})}{
n^s}\Big|^2\cr&=& 
  \sum_{j=1}^N (j^{2s}-(j-1)^{2s})\Big|\sum_{\m=j}^N  \frac{x_\mu}{\mu ^s} \Big|^2
\cr \int_\R  \Big|\sum_{n=1}^Nx_n n^{ i s t}\Big|^4\frac{\dd t}{\pi( t^2+1)}&=& \E \Big|  \sum_{n,\nu=1}^N \, x_\nu  x_n \frac{W((n\nu)^{2s}) }{(n\nu) ^s }\Big|^2\cr &=& 
  \sum_{j=1}^{N^2} (j^{2s}-(j-1)^{2s})\Big|\sum_{1\le n,\nu\le N\atop n\nu\ge j}   \frac{ x_\nu  x_n}{(n\nu) ^s } \Big|^2. 
\end{eqnarray*}  
  \end{lemma}
  
\begin{proof} 
 The first equality follows   from Lemma \ref{debut} and   (\ref{kb}).
 As to the second one, write
$W((n\nu)^{2s})=
\sum_{j=0}^{n\nu} g_j$,  where  
$g_j=W(j^{2s})-W((j-1)^{2s})$,
$j\ge 1$, we also have
\begin{eqnarray*}
 \int_\R  
\Big|\sum_{n=1}^Nx_n n^{ i s t}\Big|^4\frac{\dd t}{\pi( t^2+1)}&=&  
    \sum_{n,\nu=1}^N \sum_{m,\mu=1}^N    \overline{x _\m x} _m \, x_\nu  x_n \Big(\frac{n\nu  \wedge m\mu}{n\nu  \vee m\mu
}\Big)^s\cr & =&\sum_{n,\nu=1}^N \sum_{m,\mu=1}^N
\overline{x _\m x} _m \, x_\nu  x_n\E\Big(\frac{W((n\nu)^{2s}) }{(n\nu) ^s }\frac{W((m\mu)^{2s}) }{ (m\mu)^s }\Big)  \cr   &= &
\E \Big|  \sum_{n,\nu=1}^N \, x_\nu  x_n \frac{W((n\nu)^{2s}) }{(n\nu) ^s }\Big|^2=\E \Big|
\sum_{n,\nu=1}^N  \frac{  x_\nu  x_n}{(n\nu) ^s }\sum_{j=1}^{n\nu} g_j\Big|^2 \cr&=& 
     \E
\Big|\sum_{j=1}^{N^2} g_j\sum_{1\le n,\nu\le N\atop n\nu\ge j}   \frac{ x_\nu  x_n}{(n\nu) ^s } \Big|^2
= 
  \sum_{j=1}^{N^2}   \Big|\sum_{1\le n,\nu\le N\atop n\nu\ge j}   \frac{ x_\nu  x_n}{(n\nu) ^s } \Big|^2\E g_j^2
\cr 
 &= & 
  \sum_{j=1}^{N^2} (j^{2s}-(j-1)^{2s})\Big|\sum_{1\le n,\nu\le N\atop n\nu\ge j}   \frac{ x_\nu  x_n}{(n\nu) ^s } \Big|^2
 . 
\end{eqnarray*}
 \end{proof}
  We now need a technical lemma. 
%zzz $s<1/4$ gives simplifications: $2s-3/2<-1$.
 \begin{lemma} \label{x}
 For any $s>0$ and complex numbers   $x_j$, $j=1,\ldots , N$,
 \begin{eqnarray*} \sum_{j=1}^N (j^{2s}-(j-1)^{2s})\Big|\sum_{\m=j}^N  \frac{x_\mu}{\mu ^s} \Big|^2&\le &  \begin{cases}C_s\sum_{\m=1}^N 
 {|x_\m|^2} \m^{  3/2-2s} &\quad {\rm if}\  0<s<1/4, \cr 
C\sum_{\m=1}^N 
 |x_\m|^2  \m \log \m&\quad {\rm if}\  s= 1/4\cr 
C_s\sum_{\m=1}^N 
 |x_\m|^2  \m &\quad {\rm if}\  s> 1/4    . 
\end{cases}\end{eqnarray*}  \end{lemma}
  \begin{proof} Let $ y_\m= x_\m/\mu ^{s-1}$. By     H\"older's  inequality,
 \begin{eqnarray*}\sum_{j=1}^N (j^{2s}-(j-1)^{2s})\Big|\sum_{\m=j}^N  \frac{x_\mu}{\mu ^s} \Big|^2&=&\sum_{j=1}^N
(j^{2s}-(j-1)^{2s})\, \Big|\sum_{\m=j}^N 
\frac{y_\mu}{\mu  } \Big|^2 
\cr (\hbox{writing $\m=\mu ^{3/4}.\mu^{1/4}$})\quad &\le & \sum_{j=1}^N (j^{2s}-(j-1)^{2s}) \Big(\sum_{\m=j}^N  \frac{1}{ \m^{3/2} }  \Big)
\Big(\sum_{\m=j}^N 
\frac{|y_\m|^2}{ \m^{1/2}}  \Big)
\cr &\le & C_s\sum_{j=1}^N  \frac{(j^{2s}-(j-1)^{2s})}{  j^{1/2} }   
\Big(\sum_{\m=j}^N 
\frac{|y_\m|^2}{ \m^{1/2}}  \Big)
\cr &\le & C_s\sum_{\m=1}^N \frac{|y_\m|^2}{ \m^{1/2}}\sum_{j\le \m}   j^{2s-3/2}.
  \end{eqnarray*}   
If $0<s<1/4$, it follows that
  \begin{eqnarray*}\sum_{j=1}^N (j^{2s}-(j-1)^{2s})\Big|\sum_{\m=j}^N  \frac{x_\mu}{\mu ^s} \Big|^2  &\le & C_s\sum_{\m=1}^N 
\frac{|y_\m|^2}{ \m^{1/2}} =C_s\sum_{\m=1}^N 
 {|x_\m|^2} \m^{  3/2-2s}  .   \end{eqnarray*}    
If   $s> 1/4$,
 \begin{eqnarray*}\sum_{j=1}^N (j^{2s}-(j-1)^{2s})\Big|\sum_{\m=j}^N  \frac{x_\mu}{\mu ^s} \Big|^2 &\le & C_s\sum_{\m=1}^N 
\frac{|y_\m|^2}{ \m^{1/2}}\sum_{j\le \m}   j^{2s-3/2}\le C_s\sum_{\m=1}^N 
\frac{|y_\m|^2}{ \m^{1/2}}    \m^{2s-1/2}  
\cr &= &    C_s\sum_{\m=1}^N 
{|x_\m|^2} \m^{  3/2-2s}    \m^{2s-1/2}=  C_s\sum_{\m=1}^N 
 |x_\m|^2  \m .   \end{eqnarray*} 
And if $s=1/4$, 
\begin{eqnarray*}\sum_{j=1}^N (j^{2s}-(j-1)^{2s})\Big|\sum_{\m=j}^N  \frac{x_\mu}{\mu ^s} \Big|^2 &\le & C\sum_{\m=1}^N 
\frac{|y_\m|^2}{ \m^{1/2}}\sum_{j\le \m}   j^{-1}\le C\sum_{\m=1}^N 
\frac{|y_\m|^2\log \m}{ \m^{1/2}}    \
\cr &= &    C\sum_{\m=1}^N 
{|x_\m|^2} \m    \log \m .   \end{eqnarray*} \end{proof}
 
Indicate now how to deduce (\ref{Wilf1}). By taking $s=1/2$, $x_j=z_j/j^{1/2}$  we get in particular
\begin{eqnarray*}
  \sum_{k=1}^N  \Big|\sum_{j=k}^N  \frac{z_j}{j}  \Big|^2& \le &  C\sum_{j=1}^N       |z_j|^2 ,  
\end{eqnarray*}
hence by Lemma \ref{w},
\begin{eqnarray*} \int_\R  
\Big|\sum_{n=1}^N\frac{z_n}{n^{\frac{1}{2}(1+it)}} \Big|^2\frac{\dd t}{\pi( t^2+1)}&=&  \sum_{k=1}^N  \Big|\sum_{j=k}^N  \frac{z_j}{j}  \Big|^2  \ \le \  C\sum_{j=1}^N       |z_j|^2.
\end{eqnarray*}
Making the variable change $t=2\theta $, gives
\begin{eqnarray*}\int_\R  
\Big|\sum_{n=1}^N\frac{z_n}{n^{\frac{1}{2} +i\theta}} \Big|^2\frac{ \dd \theta}{ \pi(  \theta^2+\frac{1}{4})}&\le & 
 2C\sum_{j=1}^N       |z_j|^2,
 \end{eqnarray*} 
which is  (\ref{Wilf1}) up to the value of the constant.

%%%%%%%%%%%%%%
%%%%%%%%%%%%%% 
\subsection{Example.}
One can deduce similar estimates for integrals of power four.
\begin{eqnarray*}
C_1(\log N)^3\le \int_\R  \Big|\sum_{n=1}^N\frac{1}{n^{1+i{t}/{2}}}\Big|^4\frac{\dd t}{\pi( t^2+1)}  \le   C_2(\log N)^3. \end{eqnarray*}
 Take  $s=1/2$, $x_n=1/n$.  Then
\begin{eqnarray*}
 \int_\R  
\Big|\sum_{n=1}^N\frac{1}{n^{1+it/2 }}\Big|^4\frac{\dd t}{\pi( t^2+1)}  &= & 
  \sum_{j=1}^{N^2}  \Big|\sum_{1\le n,\nu\le N\atop n\nu\ge j}   \frac{ 1}{(n\nu)^{ \frac{3}{2}  } } \Big|^2
 . 
\end{eqnarray*}
Next 
\begin{eqnarray*}
   \sum_{1\le n,\nu\le N\atop n\nu\ge j}   \frac{ 1}{(n\nu)^{ \frac{3}{2}  } } &=&  \sum_{j< n \le N } \frac{ 1}{ n ^{ \frac{3}{2}  } } \sum_{1\le  \nu\le N }   \frac{ 1}{ \nu ^{ \frac{3}{2}  } }+  \sum_{1\le n\le j }\frac{ 1}{ n ^{ \frac{3}{2}  } } \sum_{1\le  \nu\le N\atop  \nu\ge j/n}  \frac{ 1}{ \nu ^{ \frac{3}{2}  } }\cr
   &\le  & \frac{C}{ j ^{ \frac{1}{2}  } }      +  C   \sum_{1\le  n\le j} \frac{ 1}{ n ^{ \frac{3}{2}  } }\Big( \frac{  n}{  j   }\Big)^{ \frac{1}{2}  }=\frac{C}{ j ^{ \frac{1}{2}  } }      +  C  \sum_{1\le  n\le j}  \frac{ 1}{ n j^{ \frac{1}{2}  } }\le C  \frac{ \log j}{   j^{ \frac{1}{2}  } } . 
\end{eqnarray*}
Thus 
\begin{eqnarray*}
 \int_\R  
\Big|\sum_{n=1}^N\frac{1}{n^{1+\frac{1}{2}it }}\Big|^4\frac{\dd t}{\pi( t^2+1)}  &\le  &  C
  \sum_{j=1}^{N^2}   \frac{ \log^2 j}{   j  }\le C(\log N)^3
 . 
\end{eqnarray*}
Further, for $j\le N/2$,
\begin{eqnarray*}
   \sum_{1\le n,\nu\le N\atop n\nu\ge j}   \frac{ 1}{(n\nu)^{ \frac{3}{2}  } } &\ge &    \sum_{1\le n\le j }\frac{ 1}{ n ^{ \frac{3}{2}  } } \sum_{1\le  \nu\le N\atop  \nu\ge j/n}  \frac{ 1}{ \nu ^{ \frac{3}{2}  } }\ge  C   \sum_{1\le  n\le j} \frac{ 1}{ n ^{ \frac{3}{2}  } }\Big( \frac{  n}{  j   }\Big)^{ \frac{1}{2}  }=   C\frac{ \log j}{   j^{ \frac{1}{2}  } } , 
\end{eqnarray*}
and \begin{eqnarray*}
 \int_\R  
\Big|\sum_{n=1}^N\frac{1}{n^{1+\frac{1}{2}it }}\Big|^4\frac{\dd t}{\pi( t^2+1)}  &\ge  &  C
  (\log N)^3
 . 
\end{eqnarray*}
%%%%%%%%%%%%
%%%%%%%%%%%%
\subsection{Lubinsky's space $\boldsymbol{\mathcal L}$}It is natural to consider the (Hilbert) space $\mathcal L$ consisting with all Borel-measurable functions $f:\R\to \C$ such that 
$$ \|f\| ^2 = \int_\R |f(t)|^2\frac{ \dd t}{ \pi( t^2+1)}<\infty. $$
That question was  recently investigated by Lubinsky in \cite{L1}. Let $\l_0=0$ and $1=\l_1<\l_2<\ldots $ with $\lim_{k\to \infty} \l_k=\infty$. Applying the Gram-Schmidt process to $\{\l_n^{-it}, n\ge 1\}$, produces the sequence of orthonormal Dirichlet polynomials
$$\phi_n(t) = \frac{\l_n^{1-it}-\l_n^{1-it}}{\sqrt {\l_n^2-\l_{n-1}^2}}, \qq\qq n=1,2,\ldots $$

 Let $F(t)=\sum_{n=1}^\infty a_n \l_n^{-it}$ where $\{ a_n, n\ge 1\}\subset \C$ and let $s>0$. Recall  Th. 1.1 in \cite{L1}. Assume that the series 
$$ \sum_{n=1}^\infty (\l_k^{2s}-\l_{n-1}^{2s})\Big|\sum_{n=k}^\infty \frac{a_n}{\l_n^s}\Big|^2$$
converges. Then $F(s.)\in \mathcal L$ and
\begin{equation}\label{l}\int_\R |F(st)|^2 \frac{\dd t}{\pi (1+t^2)}=\sum_{n=1}^\infty (\l_k^{2s}-\l_{n-1}^{2s})\Big|\sum_{n=k}^\infty \frac{a_n}{\l_n^s}\Big|^2.
\end{equation}
Further,  $F(s.)$ is the limit  in $\mathcal L$  of some (explicited) subsequence of its partial sums. 
%\begin{remark} This provides only a sufficient condition for $F$ to belong to $%\mathcal L$.
\vskip 2pt Consequently, in Lemma \ref{w}, we also have that
 \begin{eqnarray}\label{zd} \int_\R\Big|\sum_{n=1}^\infty x_n n^{ i s t}\Big|^2
 \frac{\dd t}{\pi( t^2+1)}&=& 
  \sum_{j=1}^\infty  (j^{2s}-(j-1)^{2s})\Big|\sum_{\m=j}^\infty  \frac{x_\mu}{\mu ^s} \Big|^2 \cr &=& \E\Big|\sum_{n=1}^\infty x_n\frac{W(n^{2s})}{
n^s}\Big|^2\
\end{eqnarray}  
provided that the Brownian  series 
$\sum_{n=1}^\infty x_n\frac{W(n^{2s})}{
n^s}$
converges in $L^2(\P)$.  
%\begin{remark} The case $x_n=n^{-\s}$, $\s>0$ corresponds to the Riemann %zeta function $\zeta(\s +ist)$. Let $\s =1/2$, $s=1/2$. Then $\zeta\big(1/2+i%(t/2) \big)\in \mathcal L$, whereas  the series in the intermediate term of  %%\eqref{zd} equals $ \sum_{j=1}^\infty  \big|\sum_{\m=j}^\infty  \frac{1}{\mu} 
%\big|^2  $ and thus obviously diverges.  Indeed, by Huxley's estimate \cite
%{r9},  $\zeta( {1\over 2}+it) ={\mathcal O}_\e(t^{32/205 +\e})$,   
%($32/205\approx 0,156097561..$), and as $\zeta$ is of finite order in any %half-plan $\s_0>0$,   the Cauchy means $$ \int_\R\Big|\zeta\big(1/2+i(t/%2)\big)\Big|^2 \frac{\dd t}{\pi( t^2+1)} $$ is convergent. \end{remark}
\vskip 2pt
 New sufficient conditions for $F$ to belong to  $\mathcal L$ can further  easily be derived from  Lemma \ref{x}. More precisely, %  for any $s>0$ and complex numbers   $\{x_j,j\ge 1\}$,
 %\begin{eqnarray*} \sup_{N\ge 1} \ \sum_{j=1}^N (j^{2s}-(j-1)^{2s})\Big|\sum_{\m=j}^N  %\frac{x_\mu}{\mu ^s} \Big|^2&\le &  \begin{cases}C_s\sum_{\m=1}^\infty 
% {|x_\m|^2} \m^{  3/2-2s} &\quad {\rm if}\  0<s<1/4, \cr 
%C\sum_{\m=1}^\infty|x_\m|^2  \m \log \m&\quad {\rm if}\  s= 1/4,\cr 
%C_s\sum_{\m=1}^\infty  |x_\m|^2  \m &\quad {\rm if}\  s> 1/4    . 
%\end{cases}\end{eqnarray*} 
%It follows in particular that we have
 \begin{corollary}Let $F(t)=\sum_{n=1}^\infty{x_n}{n^{-i t }}$ where $x_n\ge 0$ and let $s>0$. A sufficient condition for $F(s.)\in \mathcal L$ is \begin{eqnarray*} \begin{cases}  \sum_{\m=1}^\infty 
 {x_\m^2}\, \m^{  3/2-2s}<\infty &\quad {\rm if}\  0<s<1/4,\cr \sum_{\m=1}^\infty
 x_\m^2 \, \m \log \m<\infty &\quad {\rm if}\  s= 1/4,\cr 
\sum_{\m=1}^\infty 
 x_\m^2\,  \m<\infty  &\quad {\rm if}\  s> 1/4    . 
\end{cases}\end{eqnarray*} 
\end{corollary}
\begin{proof}  Under either of these conditions, the corresponding series 
$$\sum_{j=1}^\infty  (j^{2s}-(j-1)^{2s})\Big(\sum_{\m=j}^\infty  \frac{x_\mu}{\mu ^s} \Big)^2$$
is convergent, since for instance if $s>1/4$,  by Lemma \ref{x}, for all $N_0\ge 1$, for all $N\ge N_0$, 
 \begin{eqnarray*} \sum_{j=1}^{N_0} (j^{2s}-(j-1)^{2s})\Big(\sum_{\m=j}^N  \frac{x_\mu}{\mu ^s} \Big)^2&\le &  
C_s\sum_{\m=1}^N 
 x_\m^2  \,\m .\end{eqnarray*} 
%Taking $N_0=1$, and next letting $N\to \infty$ implies that the series $\sum_%{\m=1}^\infty  \frac{x_\mu}{\mu ^s} $ is convergent.
 %\begin{eqnarray*} \sum_{j=1}^{N_0} (j^{2s}-(j-1)^{2s})\Big|\sum_{\m=j}^N  
 %\frac{x_\mu}{\mu ^s} \Big|^2&\le &  \begin{cases}C_s\sum_{\m=1}^N 
 %{|x_\m|^2} \m^{  3/2-2s} &\quad {\rm if}\  0<s<1/4, \cr C\sum_{\m=1}^N 
% |x_\m|^2  \m \log \m&\quad {\rm if}\  s= 1/4\cr C_s\sum_{\m=1}^N 
% |x_\m|^2  \m &\quad {\rm if}\  s> 1/4    . \end{cases}\end{eqnarray*} 
The conclusion  thus follows from the afore mentionned Lubinsky's result.\end{proof}
%%%%%%%%%%%%%%%%%%%%
%%%%%%%%%%%%%%%%%%%%

 \subsection{Higher moments}
Let   $s\ge 0$, $r>0$.  Consider    the  
more general   integrals 
$$ \int_\R \Big|\sum_{n=1}^N\frac{x_n}{n^{is t }}\Big|^{r} \frac{\dd t}{\pi( t^2+1)},$$
and in particular,  for any positive integer $k$,
$$ I_k(N,\s,s)=\int_\R \Big|\sum_{n=1}^N\frac{1}{n^{\s+is t }}\Big|^{2k} \frac{\dd t}{\pi( t^2+1)},$$
corresponding to Dirichlet approximating polynomials. By simple    iteration, Lemmas \ref{debut} and \ref{w} extend to general integer moments.\begin{lemma}\label{debut1} For any positive integer
$q$, we have \begin{eqnarray*}  \int_\R  
 \Big|\sum_{n=1}^Nx_n n^{- i s t}\Big|^{2q}\frac{\dd t }{\pi( t^2+1)}=  \sum_{1\le \m_{1}, \ldots, \m_{q}\le N \atop 1\le  \nu_{1}, \ldots ,  \nu_{q} \le N} 
 \overline{x _{\m_{1}}\ldots  x} _{\m_{q}}\,  x _{\nu_{1}}\ldots  x _{\nu_{q}}   \Big(\frac{\nu_{1}  \ldots \nu_{q} 
 \wedge \m_{1}  \ldots \m_{q}}{\nu_{1}  \ldots \nu_{q}  \vee \m_{1}  \ldots \m_{q} }\Big)^s.
\end{eqnarray*}
 And
\begin{eqnarray*}  \int_\R  
 \Big|\sum_{n=1}^Nx_n n^{- i s t}\Big|^{2q}\frac{\dd t }{\pi( t^2+1)}&=& \E \Big|\sum_{1\le  \nu_{i} \le N\atop 1\le i\le q } x _{\nu_{1}}\ldots  x
_{\nu_{q}}\frac{W(( {\nu_{1}}\ldots   {\nu_{q}})^{2s}) }{ ( {\nu_{1}}\ldots   {\nu_{q}})^{ s}  }\Big|^2
\cr &=& \sum_{j=1}^{N^q} (j^{2s}-(j-1)^{2s})\Big|\sum_{{1\le  \nu_{i} \le N\atop 1\le i\le q }\atop \nu_{1}\ldots\nu_{q}\ge j} 
\frac{ x _{\nu_{1}}\ldots  x _{\nu_{q}}}{(\nu_{1}\ldots\nu_{q}) ^s } \Big|^2.
\end{eqnarray*} 
  \end{lemma} 
We omit the proof.   
By \eqref{l} and the considerations made after, it also follows that  
\begin{corollary}Let $q$ be a positive integer, $s>0$ and let
$F^q(t)= \big(\sum_{n=1}^\infty x_n n^{- i  t}\big)^q$. 
$$ F^q(s.)\in \mathcal L  \qq \Big(\hbox{\it and thus}\qq  \int_\R  
 \Big|\sum_{n=1}^\infty x_n n^{- i s t}\Big|^{2q}\frac{\dd t }{\pi( t^2+1)}<\infty\Big)$$
if the Brownian sum
$$ \sum_{1\le  \nu_{i} \le N\atop 1\le i\le q } x _{\nu_{1}}\ldots  x
_{\nu_{q}}\frac{W(( {\nu_{1}}\ldots   {\nu_{q}})^{2s}) }{ ( {\nu_{1}}\ldots   {\nu_{q}})^{ s}  }$$
converges in $L^2(\P)$.
\end{corollary}
%%%%%%%%%%%%%%%%%%%%
%%%%%%%%%%%%%%%%%%%%

\subsection{Connection with mean-values of Dirichlet polynomials}    Of first importance in the previous formulas is the role played by the parameter $s$, and more precisely the behavior of the Cauchy means when $s\to \infty$.

Lubinsky has established a clarifying  link with mean-values of general Dirichlet polynomials. 
We state it under slightly weaker assumptions than in \cite{L1} p.~428. 

\begin{lemma} Let $g:\R\to \C$ and define formally for any $s\ge 0$, $\mathcal M(s)= \frac{1}{2s}\int_{-s}^s |g(t)| \dd t $. Then,
\begin{equation}\label{equiv0} 
 \int_{+0} |g(t)|\log \frac{1}{t}\, \dd t<\infty\quad \Longleftrightarrow\quad \int_{+0} \mathcal M(s)\, \dd s<\infty. 
\end{equation}
\begin{equation}\label{equiv} 
\bigg(\int_{\R}   \frac{|g(t)|}{1+t^2}\, \dd t <\infty\ \ {\it and} \ \ \int_{+0} |g(t)|\log \frac{1}{t}\, \dd t<\infty\bigg)\quad \Longleftrightarrow\quad \int_{\R}   \frac{\mathcal M(s)}{1+s^2}\, \dd s<\infty. 
\end{equation}
Under any of the previous properties, we further have
\begin{equation}\label{comp0}\int_\R |g(st)| \frac{\dd t}{\pi (1+t^2)}=4\int_0^\infty\mathcal M(su)\frac{u^2}{\pi (1+u^2)^2} \dd u.\end{equation}
And if moreover, $\mathcal M(s)$ is  locally bounded,
%bounded on bounded intervals, 
then
 \begin{equation}\label{compa}\lim_{s\to \infty}\int_\R |g(st)| \frac{\dd t}{\pi (1+t^2)}=\lim_{s\to \infty} \mathcal M(s),\end{equation}
 if the preceding limit exists and is finite.
   \end{lemma} 
   
     \begin{proof} Assertion \eqref{equiv0} follows by integration by part.
Further,  if $\eta>0$, $$ \int_{\eta\le |t|<\infty}   \frac{|g(t)|}{1+t^2}\, \dd t <\infty \quad \Longleftrightarrow\quad \int_{\eta\le s <\infty}   \frac{\mathcal M(s)}{1+s^2}\, \dd s<\infty.$$
Hence \eqref{equiv}  follows. An integration by part gives \eqref{comp0}.   Since $ \mathcal M(s)\to \l$, say, and $|\l|<\infty$, there is a real $A>0$ and a real $Y>0$ such that we have $| \mathcal M(y)|\le A$ if $y\ge Y$.  By assumption, $\mathcal M(s)$ is  locally bounded, we also have $ \mathcal M(y)\le B$ if $0\le y\le Y$. Thus
$ \mathcal M(y)\le A\vee B$ on $\R_+$. 
Therefore
$$ \frac{ \mathcal M(su)u^2}{ (1+u^2)^2} \le \frac{(A\vee B)u^2}{ (1+u^2)^2}\in L^1(\R_+) .$$
And  \eqref{compa} follows from the  dominated convergence theorem.\end{proof}

% Let $g:\R\to \C$ be continuous with 
%\marginpar{[$g$ or $g^2$]}
%\marginpar{[$\int_\R   \frac{\mathcal M(u) }{1+u^2}\dd u <\infty$ and 
%$\int_\R   \frac{|g(t)|^2 }{1+t^2}\dd t <\infty$ suffice]} 
%\marginpar{[$\int_\R   \frac{\mathcal M(u) }{1+u^2}\dd u <\infty\Leftrightarrow %\int_\R   \frac{|g(t)|^2 }{1+t^2}\dd t <\infty$]}$$\mathcal M(s)= \frac{1}{2s}\int_{-%s}^s |g(t)|^2 \dd t = o(s), \qq s\to \infty . $$
% Plainly
%\marginpar{[To apply the dominated convergence th. requires that $\mathcal %M(y)= o(y^\a)$ with $\a<1$]}  \begin{equation}\label{l2}\int_\R |g(st)|^2 \frac
%{\dd t}{\pi (1+t^2)}= \frac{4}{\pi} \int_0^\infty 
%\mathcal M(su) \frac{u^2 }{\pi (1+u^2)^2}\dd u.\end{equation}
 
  Letting $g=\big|\sum_{n=1}^\infty x_n n^{- i  t}\big|^{2q}$, where $q$ is a positive integer yields
\begin{equation*}\lim_{s\to \infty}\int_\R  
 \Big|\sum_{n=1}^\infty x_n n^{- i s t}\Big|^{2q}\frac{\dd t }{\pi( t^2+1)}=\lim_{s\to \infty} \frac{1}{2s}\int_{-s}^s \Big|\sum_{n=1}^\infty x_n n^{- i  t}\Big|^{2q} \dd t ,\end{equation*}
 provided that the second limit exists. 
% \vskip 2 pt It is now immediate that the asymptotic behavior (as $s\to \infty$)  %of the Brownian sum process ${\bf Y}= \{ {\bf Y}(s), s\in \R_+\}$ defined by 
%$$ {\bf Y}(s)=\sum_{ \nu_{i}\ge 1\atop 1\le i\le q } x _{\nu_{1}}\ldots  x
%_{\nu_{q}}\frac{W(( {\nu_{1}}\ldots   {\nu_{q}})^{2s}) }{ ( {\nu_{1}}\ldots   {\nu_
%{q}})^{ s}  }$$ plays a critical role in the present context.  
 \vskip 3 pt 
 
Another link with standard mean-values of Dirichlet sums is provided with the next lemma.
  \begin{lemma}\label{comp}Let $q,S,T$ be positive reals. Then 
\begin{eqnarray*}
  \int_{S}^{\sqrt{S^2+T^2}} \!\!\bigg( \int_{\R} \Big|\sum_{n =1}^N\frac{x_n}{ n^{ ist}}\Big|^{2q}\frac{\dd t}{\pi( t^2+1)}\bigg) \dd s 
= 
\frac{1}{2\pi }\int_{\R} \Big|\sum_{n =1}^N\frac{x_n}{ n^{ i\theta }}\Big|^{2q} \log  \Big( 1+ \frac{T^2}{\theta^2+S^2}  
\Big)\dd \theta.
\end{eqnarray*}
 Moreover,\begin{eqnarray*}
 \frac{1}{S}\int_{0}^S \Big|\sum_{n =1}^N\frac{x_n}{ n^{ i\theta }}\Big|^{2q}    \dd \theta  &\le  &\Big(\frac{2\pi }{\log 2} \Big)
\sup_{S\le s\le 2S}\,   \int_{\R}
\Big|\sum_{n =1}^N\frac{x_n}{ n^{ ist}}\Big|^{2q}\frac{\dd t}{\pi( t^2+1)}   
  . \end{eqnarray*}
\end{lemma}
This provides a partial  converse to Lubinsky's observation. Indeed, assume that 
$$ \limsup_{s\to \infty } \int_{\R}
\Big|\sum_{n =1}^N\frac{x_n}{ n^{ ist}}\Big|^{2q}\frac{\dd t}{\pi( t^2+1)} = \l.$$
 Then it follows from the second part of the Lemma that
$$ \lim_{S\to \infty } \frac{1}{S}\int_{0}^S \Big|\sum_{n =1}^N\frac{x_n}{ n^{ i\theta }}\Big|^{2q}    \dd \theta \le \Big(\frac{2\pi }{\log 2} \Big) \l.$$
\begin{proof}
By  using the  variable change
$t=\theta/s$, we get
\begin{eqnarray*}
& &   \int_{S}^{\sqrt{S^2+T^2}} \bigg( \int_{\R} \Big|\sum_{n =1}^N\frac{x_n}{ n^{ ist}}\Big|^{2q}\frac{\dd t}{\pi( t^2+1)}\bigg) \dd s 
\cr &=
& 
\int_{S}^{\sqrt{S^2+T^2}} \bigg(   \int_{\R} \Big|\sum_{n =1}^N\frac{x_n}{ n^{ i\theta }}\Big|^{2q}\frac{s\dd \theta}{\pi( \theta^2+s^2)}\bigg) \dd s
  \cr &= & 
\int_{\R} \Big|\sum_{n =1}^N\frac{x_n}{ n^{ i\theta }}\Big|^{2q}  \bigg(  \int_{S}^{\sqrt{S^2+T^2}}       \frac{s \dd s}{\pi( \theta^2+s^2)} \bigg)\dd
\theta
\cr &= & 
\frac{1}{2\pi }\int_{\R} \Big|\sum_{n =1}^N\frac{x_n}{ n^{ i\theta }}\Big|^{2q} \log  \Big( 1+ \frac{T^2}{\theta^2+S^2}  \Big)\dd \theta.
\end{eqnarray*}
 
Letting $T=\sqrt{2}S$ gives
\begin{eqnarray*}
  \int_{S}^{2S} \bigg( \int_{\R} \Big|\sum_{n =1}^N\frac{x_n}{ n^{ ist}}\Big|^{2q}\frac{\dd t}{\pi( t^2+1)}\bigg) \dd s  &= & 
\frac{1}{2\pi }\int_{\R} \Big|\sum_{n =1}^N\frac{x_n}{ n^{ i\theta }}\Big|^{2q} \log  \Big( 1+ \frac{2S^2}{\theta^2+S^2}  \Big)\dd \theta
\cr &\ge  & 
\frac{\log 2}{2\pi }\int_{0}^S \Big|\sum_{n =1}^N\frac{x_n}{ n^{ i\theta }}\Big|^{2q}    \dd \theta. \end{eqnarray*}
Therefore
\begin{eqnarray*}
 \int_{0}^S \Big|\sum_{n =1}^N\frac{x_n}{ n^{ i\theta }}\Big|^{2q}    \dd \theta&\le & \frac {2\pi }{\log 2} \int_{S}^{2S} \bigg(
\int_{\R}
\Big|\sum_{n =1}^N\frac{x_n}{ n^{ ist}}\Big|^{2q}\frac{\dd t}{\pi( t^2+1)}\bigg) \dd s \cr &\le  & 
 \Big(\frac {2\pi S }{\log 2} \Big) \, \sup_{S\le s\le 2S} \
\int_{\R}
\Big|\sum_{n =1}^N\frac{x_n}{ n^{ ist}}\Big|^{2q}\frac{\dd t}{\pi( t^2+1)}   . \end{eqnarray*} 
\end{proof} 
 \vskip 3 pt
 Finally, a  simple re-summation argument also provides a direct connection with standard mean-values of Dirichlet polynomials. 
\begin{lemma} \label{comp}There exist two positive absolute constants $c,
C$ such that
$$c\, \sum_{j=1}^{\infty}\frac{\mathcal M_{j}}{j^2}\le  \int_\R \Big|\sum_{n=1}^N\frac{1}{n^{\s+is t }}\Big|^{r} \frac{\dd t}{\pi( t^2+1)} \le C\, 
\sum_{j=1}^{\infty}\frac{\mathcal M_{j}}{j^2}.$$ where we set 
$$ \mathcal M_{j}= \frac{1}{2j}\int_{-j}^{j} \big|\sum_{n=1}^N\frac{1}{n^{\s+is t }}\big|^{r}\dd t, \qq j=1,2,\ldots$$
\end{lemma}
 \begin{remark}It is well-known that for any   complex numbers  $ x_1, \ldots, x_N$  and  any $0<\a<\infty$, the limits  
$$  \lim_{T\to \infty}  \frac{1}{T} \int_{0}^T \big| \sum_{n=1}^N x_n n^{-it}\big|^\a \dd t  $$
exist. The series $\sum_{j=1}^{\infty}\frac{\mathcal M_{j}}{j^2}$ is thus convergent. For the values $\a=2k$, $k=1,2,\ldots$, we recall that 
\begin{equation*}\lim_{T\to\infty}{1\over 2T}\int_{-T}^T
\Big|  \sum_{n=1}^N  {1\over  n^{ \s+it }}\Big|^{2k}\dd t =\sum_{1\le m\le N^k} {d_{k,N }^2(m)\over m^{2\s}},\end{equation*} 
 where $d_{k,N }(m)$ denotes the number of representations of  $m$ as a product of
$k$ factors less or equal to $N$.
  \end{remark}
  
 \begin{proof}[Proof of Lemma \ref{comp}]
Let 
$$u_k=\int_{k-1}^{k} \big|\sum_{n=1}^N\frac{1}{n^{\s+is t }}\big|^{2k}\dd t, \qq \quad k=1,2,\ldots$$ and note that 
$$ \sum_{k=0}^\infty  \frac{u_k}{\pi( k^2+1)}\le \int_0^\infty \Big|\sum_{n=1}^N\frac{1}{n^{\s+is t }}\Big|^{2k} \frac{\dd t}{\pi( t^2+1)}\le \sum_{k=0}^\infty  \frac{u_k}{\pi( (k-1)^2+1)}, $$
   Let $D_j=\sum_{k=1}^j u_k$, $j\ge 1$.
By applying Abel summation 
 $$\sum_{k=1}^r u_ky_k= D_ry_{r+1}+ \sum_{j=1}^{r}D_j(y_j-y_{j+1}),$$  with $ y_k= \frac{1}{\pi( (k-1)^2+1)}$, we get  
$$ \int_0^r \Big|\sum_{n=1}^N\frac{1}{n^{\s+is t }}\Big|^{2k} \frac{\dd t}{\pi( t^2+1)}\le \frac{1}{\pi( r^2+1)} \int_{0}^{r} \big|\sum_{n=1}^N\frac{1}{n^{\s+is t }}\big|^{2k}\dd t$$
$$ + \sum_{j=1}^{r}\frac{2j-1}{\pi( (j-1)^2+1)(j^2+1)}\int_{0}^{j} \big|\sum_{n=1}^N\frac{1}{n^{\s+is t }}\big|^{2k}\dd t. $$
Hence
$$ \int_0^\infty \Big|\sum_{n=1}^N\frac{1}{n^{\s+is t }}\Big|^{2k} \frac{\dd t}{\pi( t^2+1)}\le C \sum_{j=1}^{\infty}\frac{1}{j^3}\int_{0}^{j} \big|\sum_{n=1}^N\frac{1}{n^{\s+is t }}\big|^{2k}\dd t. $$
 Operating similarly for the lower part and next for the integration over $\R_-$ provides the claimed estimate.\end{proof}
%           \begin{proposition}\label{main1}For all $T>0$,
%$$\int_\R  
%\Big|\sum_{n=1}^N\frac{1 }{n^{ \s +i st}}\Big|^{2q} \frac{\dd t }{\pi( t^2+1)} \le C_q 
%\Big(\frac{M^{2s+1}}{M^{s}-1}\Big) N^q\ {1\over 2T}
%\int_{| t |\le T}   \Big|\sum_{n=1}^N {1\over n^{\s+i t}} \Big|^{2q}\dd  t.$$
%$$"\int_\R  
%\Big|\sum_{n=1}^N\frac{1 }{n^{ \s +i st}}\Big|^{2q} \frac{\dd t }{\pi( t^2+1)} \le C_q 
%\Big(\frac{M^{2s+1}}{M^{s}-1}\Big) N^q\sum_{1\le m\le N^q} {d_{q,N }^2(m)\over %m^{2\s}}"$$\end{proposition}

%%%%%%%%%%%
%%%%%%%%%%%
In the next subsection, we  investigate   the behavior of Cauchy
integrals when the parameter $s$ is small and the moments are high.
\subsection{Behavior of $\boldsymbol{I_{k}(N,\s,s)}$ for $\boldsymbol{ s {=} s(k)}$ small and $\boldsymbol{ k}$ large}
 We now consider the behavior of these
integrals when $s$ and
$k$ are simultaneously varying. More precisely, we will study the case when $s=1/\sqrt {c_{\s,N} k}$ where $c_{\s,N}\sim c$  as $k\to
\infty$ ($c=c(\s)$ will be an explicit positive constant). 
\vskip 2 pt
We obtain the following  very precise uniform estimate.
\begin{theorem} \label{sk}There exist two positive numerical constants
$c_0,C$ such that for all positive integers
$N$, $k$   and $0\le \s<1$,
\begin{eqnarray*} \Big|\int_\R \big|\sum_{n=1}^N\frac{1}{ n^{\s+i t/\sqrt{c_{\s,N}  k}   }  }\big|^{2k} \frac{\dd t}{\pi( t^2+1)}  - c_0\big(\sum_{n=1}^N\frac{1}{n^\s
}\big)^{2k}   \Big|  
 &\le  & C\ \frac{ (1-\s)\log N}{k^{1/2}}\Big(\sum_{n=1}^N\frac{1}{n^\s
}\Big)^{2k}
 \end{eqnarray*}
where $$c_{\s,N}=\frac{2}{(1-\s)^2} +  \mathcal O( N^{\s-1}(\log N)^2).$$\end{theorem}
%\begin{remark} We actually don't know whether it is possible to similarly estimate %the integral \begin{eqnarray*} \int_\R \big|\sum_{n=1}^N\frac{1}{n^{\s+i t(1-\s)/\sqrt%{2 k}}}}\big|^{2k} \frac{\dd t}{\pi( t^2+1)} . \end{eqnarray*}\end{remark}

%[zzz]By H\"older's inequality, for $\k\ge k\ge 1$,
%\begin{eqnarray*} \int_\R \big|\sum_{n=1}^N\frac{1}{n^{\s+i t/s_\k}}\big|^{2k} \frac
%{\dd t}{\pi( t^2+1)} &\le & \Big(\int_\R \big|\sum_{n=1}^N\frac{1}{n^{\s+i t/s_\k}}\big|%^{2\k} \frac{\dd t}{\pi( t^2+1)}\Big)^{k/\k} \cr &\le & \Big( \big(\sum_{n=1}^N\frac{1}
%{n^\s}\big)^{2\k}\big\{ c_0+ C\ \frac{ (1-\s)\log N}{k^{1/2}}\big\}\Big)^{k/\k} .
% \end{eqnarray*}
%%%%%%%%%%%%%%%%%%%%
\section{Proof of Theorem \ref{sk}}
%%%%%%%%%%%%%%%%%%%%
%%%%%%%%%%%%%%%%%%%%
 Our proof is probabilistic.  We introduce a random model and first establish an interesting property (Lemma \ref{var}) of
this one. We don't know whether this model  has been investigated somewhere.
\subsection{A random model}  Let $\s\ge 0$.   Let $N$ be some
positive integer and note $L_N= \sum_{n=1}^{ N} \frac{1}{n^\s}$.  Let $Y$    be random variable
  defined   by 
\begin{equation}\label{iid}\P\{ Y  =  \log n\}=\frac{1}{n^\s L_N},\ n=1,\ldots, N.
 \end{equation}
Let $Y_1, \ldots, Y_k$ be independent copies of $Y$ and note $S_k= Y_1+\ldots+ Y_k$.
\begin{lemma} \label{rm} Let  $\widetilde S_k$ denote  a symmetrization of $S_k$. Then,
\begin{eqnarray*} \Big|\sum_{n=1}^N\frac{1}{n^{\s+it}}\Big|^{2k} &=&
\Big(\sum_{n=1}^N\frac{1}{n^\s }\Big)^{2k}\, \E e^{ it \widetilde S_k}
  .\end{eqnarray*}  
\end{lemma}  
\begin{proof}
We indeed have\begin{eqnarray*} \P\{S_k=   \log m\}= \sum_{1\le n_1,\ldots, n_k\le  N\atop n_1 \ldots  n_k=m} \P
\big\{Y_1= \log n_1,\ldots,Y_k= \log n_k\big\}\,  
=\frac{ \d_{k,N} (m)}{m^\s L_N^k} 
 \end{eqnarray*} 
where we set $ \d_{k,N} (m)=\#\big\{(n_1,\ldots, n_k)\in \{1,  N\}^k : m= n_1 \ldots  n_k \big\}$. 

 Further
\begin{eqnarray*}\Big|\sum_{n=1}^N\frac{1}{n^{\s+ist}}\Big|^{2k}&=& \Big(\sum_{m=1}^N  {1\over
m^{ \s+ist  } }\Big)^{k} \Big(\sum_{n=1}^N  {1\over
n^{ \s-ist  } }\Big)^{k}\,=\,\Big(\sum_{\m=1}^{N^k}{\d_{k,N} (\m)\over \m^{ \s+ist  } }\Big)\Big(\sum_{\nu=1}^{N^k}{\d_{k,N} (\nu)\over \nu^{ \s-ist  }
}\Big)
\cr &=& L_N^{2k}\Big(\sum_{\m=1}^{N^k}{\P\{S_k=  \log \m\}\over \m^{ ist  } }\Big)\Big(\sum_{\nu=1}^{N^k}{\P\{S_k=   \log \nu\}\over \nu^{-ist  } }\Big)
\cr &=& L_N^{2k}\Big|\sum_{\m=1}^{N^k}{\P\{S_k= \log \m\}\over \m^{ist  } }\Big|^2
\ =\   L_N^{2k}\big|\E e^{-ist S_k}\big|^2 \, =\,  L_N^{2k}\E e^{ ist  \widetilde S_k}.
\end{eqnarray*}
  Hence, 
\begin{eqnarray*}\Big|\sum_{n=1}^N\frac{1}{n^{\s+ist}}\Big|^{2k} &=&
\Big(\sum_{n=1}^N\frac{1}{n^\s }\Big)^{2k}\E e^{ ist \widetilde S_k}.
\end{eqnarray*}
\end{proof}
 \begin{lemma} \label{rm1} We have the relations
\begin{eqnarray*} \frac{1}{T}\int_{-T}^{ T}\big|\sum_{n=1}^N\frac{1}{n^{\s+it}}\big|^{2k} \dd t&=& 
  \big(\sum_{n=1}^N\frac{1}{n^\s}\big)^{2k} \E  \, \frac{ \sin T\widetilde S_k }{T\widetilde S_k}
 \cr \int_\R \big|\sum_{n=1}^N\frac{1}{n^{\s+ist}}\big|^{2k} \frac{\dd t}{\pi( t^2+1)}     &=& \big(\sum_{n=1}^N\frac{1}{n^\s }\big)^{2k}
 \E
\,  e^{ -s|\widetilde S_k|} .\end{eqnarray*}\end{lemma}  
\begin{proof} By Fubini's theorem, \begin{eqnarray*} \label{rep1}\frac{1}{T}\int_{-T}^{ T}\big|\sum_{n=1}^N\frac{1}{n^{\s+ist}}\big|^{2k} \dd
t\  = \
\big(\sum_{n=1}^N\frac{1}{n^\s }\big)^{2k} \frac{1}{T}\E \int_{-T}^{ T}e^{ ist \widetilde S_k} \dd t
\  = \
  \big(\sum_{n=1}^N\frac{1}{n^\s}\big)^{2k} \E  \, \frac{ \sin sT\widetilde S_k }{ sT\widetilde S_k}
  .
\end{eqnarray*}

It also follows by   integrating that
\begin{eqnarray*} \int_\R \big|\sum_{n=1}^N\frac{1}{n^{\s+ist}}\big|^{2k} \frac{\dd t}{\pi( t^2+1)}   &=&
\big(\sum_{n=1}^N\frac{1}{n^\s }\big)^{2k} \int_\R  \E e^{ ist \widetilde S_k} \frac{\dd t}{\pi( t^2+1)}
\cr   &=&
\big(\sum_{n=1}^N\frac{1}{n^\s }\big)^{2k} \E \int_\R   e^{ ist \widetilde S_k} \frac{\dd t}{\pi( t^2+1)}\cr   &=& \big(\sum_{n=1}^N\frac{1}{n^\s }\big)^{2k}
 \E
\,  e^{ -s|\widetilde S_k|}  .
 \end{eqnarray*}

 A interesting fact of this model is that the  variance of $\widetilde Y$ is small (almost constant). This is made precise in  the  lemma below.
\begin{lemma} \label{var} Let $0<\s<1$. We have 
 \begin{eqnarray*} \E\, \widetilde Y^2=2\Big\{\sum_{m=1}^N \frac{( \log m)^2}{m^\s L_N}- \Big( \sum_{m=1}^N \frac{( \log m)}{m^\s L_N}  \Big)^2\Big\}\ =\ \frac{2}{(1-\s)^2} +  \mathcal O( N^{\s-1}(\log N)^2).\end{eqnarray*}
\end{lemma} 
It will follow from the proof that the almost constant behavior of the variance arises from cancellation of auxiliary  sums. 
\begin{proof} We use Euler-Maclaurin formula. Let
$h:[1,N]\to
\R$ be a twice differentiable function. Then
\begin{equation}\label{EM}
\sum_{k=1}^N h(k)= \int_1^N h(t) dt + \frac 12
(h(1)+h(N)) + \sum_{k=1}^{N-1} \int_0^1 \frac {t-t^2}{2} h''(k+t) dt.
\end{equation}
Applying this to $h(t)=t^\a$,  $-1<\a<0$, we get
 \begin{eqnarray*} 
\sum_{k=1}^N k^\a &= &\frac {N^{\a+1}} {\a+1} +
    \mathcal O(\left( N^\a\right))  + \left(\frac 12- \frac 1{\a+1} \right)
  +  \a(\a-1) \sum_{k=1}^{\infty} \int_0^1
       \frac {t-t^2}{2} (k+t)^{\a-2} dt
    \\
& &      -  \sum_{k=N}^{\infty}  \mathcal O\left(k^{\a-2}\right) \\
&=&  
   \frac{N^{\a+1}}{\a+1} + C_\a +  \mathcal O\left(N^{\a}\right),
\end{eqnarray*}
where
\[C_\a= \frac 12- \frac 1{\a+1} + \a(\a-1)
       \sum_{k=1}^{\infty} \int_0^1
       \frac {t-t^2}{2} (k+t)^{\a-2} dt \ .\]
Thus 
\[ L_N=\frac{N^{1-\s}}{1-\s} + C_{-\s} +  \mathcal O\left(N^{-\s}\right)\]
Apply it now to $h(t)=(\log t) t^{-\s}$. We get
\[\sum_{k=1}^N\frac{\log k}{k^\s}= \frac{N^{1-\s}\log N}{1-\s}- \frac{N^{1-\s}-1}{(1-\s)^2}+  \frac{N^{-\s}\log N}{2}+ C'_\s.\]
Next
\[\sum_{k=1}^N\frac{(\log k)^2}{k^\s}= \frac{N^{1-\s}(\log N)^2}{1-\s}- \frac{2N^{1-\s}(\log N)}{(1-\s)^2}+  \frac{2(N^{1-\s}-1)}{(1-\s)^3}+  \frac{N^{-\s}(\log N)^2}{2}+C'_\s.\]
We moreover have 
\[ \frac{1}{L_N}= \frac{1}{\frac{N^{1-\s}}{1-\s}+C_\s}+\mathcal O(N^{-2+\s}).\]
Therefore,
\begin{eqnarray} \nonumber\sum_{k=1}^N\frac{(\log k)^2}{L_Nk^\s}&=&\Big(\frac{1}{\frac{N^{1-\s}}{1-\s}+C_\s}\Big)\sum_{k=1}^N\frac{(\log k)^2}{k^\s}+ \mathcal O(N^{-1}(\log N)^2).
\end{eqnarray}
Now 
\begin{eqnarray} \nonumber
\Big(\frac{1}{\frac{N^{1-\s}}{1-\s}+C_\s}\Big)\sum_{k=1}^N\frac{(\log k)^2}{k^\s}
 & =&\Big(\frac{1}{\frac{N^{1-\s}}{1-\s}+C_\s}\Big)\Big\{ \frac{N^{1-\s}(\log N)^2}{1-\s}- \frac{2N^{1-\s}(\log N)}{(1-\s)^2}\cr & &\quad+  \frac{2(N^{1-\s}-1)}{(1-\s)^3}+  \frac{N^{-\s}(\log N)^2}{2}+C'_\s\Big\}
 \cr & =&\Big(\frac{1}{1+C_\s N^{\s-1}}\Big)\Big\{ (\log N)^2- \frac{2\log N}{1-\s}+  \frac{2}{(1-\s)^2} \cr & &\ -\frac{2N^{\s-1}}{(1-\s)^2}+  \frac{N^{-1}(\log N)^2}{2}+(1-\s)C'_\s N^{\s-1} \Big\}
 \cr & =&\Big(\frac{1}{1+C_\s N^{\s-1}}\Big)\Big\{ (\log N)^2- \frac{2\log N}{1-\s}+  \frac{2}{(1-\s)^2} \cr & &+\mathcal O( N^{\s-1}) \Big\}.\end{eqnarray}
We have
$$ 1- \frac{1}{1+C_\s N^{\s-1}}= \mathcal O( N^{\s-1}).$$
Therefore
\begin{eqnarray} \nonumber
& &\Big(\frac{1}{\frac{N^{1-\s}}{1-\s}+C_\s}\Big)\sum_{k=1}^N\frac{(\log k)^2}{k^\s}
  \cr  & =&\Big(1+ \mathcal O( N^{\s-1})\Big)\Big\{ (\log N)^2- \frac{2\log N}{1-\s}+  \frac{2}{(1-\s)^2}  +\mathcal O( N^{\s-1})\Big\}
   \cr &=&  (\log N)^2- \frac{2\log N}{1-\s}+  \frac{2}{(1-\s)^2} +  \mathcal O( N^{\s-1}(\log N)^2)+ \mathcal O( N^{2(\s-1)})
    \cr &=&  (\log N)^2- \frac{2\log N}{1-\s}+  \frac{2}{(1-\s)^2} +  \mathcal O( N^{\s-1}(\log N)^2).\end{eqnarray}
By reporting we get
\begin{eqnarray} \nonumber\sum_{k=1}^N\frac{(\log k)^2}{L_Nk^\s}&=&\ (\log N)^2- \frac{2\log N}{1-\s}+  \frac{2}{(1-\s)^2} +  \mathcal O( N^{\s-1}(\log N)^2).
\end{eqnarray}

Similarly, 
\begin{eqnarray} \nonumber\sum_{k=1}^N\frac{\log k}{L_Nk^\s}&=&\Big(\frac{1}{\frac{N^{1-\s}}{1-\s}+C_\s}\Big)\sum_{k=1}^N\frac{\log k}{k^\s}+ \mathcal O(N^{-1}\log N)
\cr &=&\Big(\frac{1}{\frac{N^{1-\s}}{1-\s}+C_\s}\Big)\Big\{ \frac{N^{1-\s}\log N}{1-\s}- \frac{N^{1-\s}-1}{(1-\s)^2}+  \frac{N^{-\s}\log N}{2}+ C'_\s\Big\}\cr & &\quad  +\mathcal O(N^{-1}\log N)
.
\end{eqnarray}
Further
\begin{eqnarray*}& &\Big(\frac{1}{\frac{N^{1-\s}}{1-\s}+C_\s}\Big)\Big\{ \frac{N^{1-\s}\log N}{1-\s}- \frac{N^{1-\s}-1}{(1-\s)^2}+  \frac{N^{-\s}\log N}{2}+ C'_\s\Big\}
\cr &=& \Big(\frac{1}{1+C_\s(1-\s)N^{\s-1}}\Big)\Big\{ \log N- \frac{1}{1-\s}+\frac{N^{\s-1}}{1-\s}+  \frac{N^{-1}\log N}{2}
\cr & & \quad + C'_\s N^{\s-1}(1-\s)\Big\}
\cr &=& \Big(\frac{1}{1+C_\s(1-\s)N^{\s-1}}\Big)\Big\{ \log N- \frac{1}{1-\s}+  \mathcal O( N^{\s-1})\Big\}
\cr &=& \Big(1+\mathcal O( N^{\s-1})\Big)\Big\{ \log N- \frac{1}{1-\s}+  \mathcal O( N^{\s-1})\Big\}
\cr &=& \log N- \frac{1}{1-\s}+ \mathcal O( N^{\s-1}\log N).
\end{eqnarray*}
Consequently,
\begin{eqnarray*} \frac12 \, \E (\widetilde Y)^2&=&\sum_{m=1}^N \frac{( \log m)^2}{m^\s L_N}- \Big( \sum_{m=1}^N \frac{ \log m}{m^\s L_N}  \Big)^2
\cr &=&  (\log N)^2- \frac{2\log N}{1-\s}+  \frac{2}{(1-\s)^2} +  \mathcal O( N^{\s-1}(\log N)^2)\cr & & \quad -\Big(\log N- \frac{1}{1-\s}+ \mathcal O( N^{\s-1}\log N)\Big)^2
\cr &=  &\frac{1}{(1-\s)^2} +  \mathcal O( N^{\s-1}(\log N)^2).\end{eqnarray*}
\end{proof}

\subsection{Proof of Theorem \ref{sk}} 
 It follows from the previous Lemma that
$$s_k^2=\E (\widetilde S_k)^2=k \E (\widetilde Y)^2=\frac{2k}{(1-\s)^2} +  \mathcal O( kN^{\s-1}(\log N)^2).$$
Choose $s=1/s_k$. Let  $g$ be a Gaussian standard random variable. Then,
\begin{eqnarray*} \int_\R \big|\sum_{n=1}^N\frac{1}{n^{\s+i t/s_k}}\big|^{2k} \frac{\dd t}{\pi( t^2+1)}       &=& \big(\sum_{n=1}^N\frac{1}{n^\s
}\big)^{2k}\big\{\E \,  e^{ -  |g|} + \E \,  e^{ - |\widetilde S_k|/s_k}-\E \,   e^{ -  |g|}   \big\} .
 \end{eqnarray*}
Hence
\begin{eqnarray*} & &\Big|\int_\R \big|\sum_{n=1}^N\frac{1}{n^{\s+i t/s_k}}\big|^{2k} \frac{\dd t}{\pi( t^2+1)}  - \big(\sum_{n=1}^N\frac{1}{n^\s
}\big)^{2k} \E \,  e^{ -  |g|}  \Big| \cr &=& \big(\sum_{n=1}^N\frac{1}{n^\s
}\big)^{2k}\big|  \E \,  e^{ - |\widetilde S_k|/s_k}-\E \,   e^{ -  |g|}   \big| .
 \end{eqnarray*}
By   the transfert formula,
\begin{eqnarray*} \big|  \E \,  e^{ - |\widetilde S_k|/s_k}-\E \,   e^{ -  |g|}   \big|   
& =&\Big|\int_0^1  \Big( \P \big\{e^{ - |\widetilde S_k|/s_k}>x\big\}- \P \big\{e^{ -  |g|}>x\big\} \Big)\dd x\Big|
\cr \  (x=e^{- y^2})\quad &=&2  \Big|\int_0^\infty  \Big( \P \big\{|\widetilde S_k|/s_k<y\big\}-\P \big\{| g|<y\big\}\Big)
ye^{- y^2} \dd y\Big|
     \cr &\le   &     4\sup_{x\in \R} \Big|\P
\big\{|\widetilde S_k|/s_k<x\big\} - \P\{|g|<x\} \Big|
 \cr &\le & A\ \frac{\E |\widetilde Y|^3}{k^{1/2}(\E |\widetilde Y|^2)^{3/2}} ,
\end{eqnarray*}
where we used  Berry--Esseen theorem's  in the last inequality,  $A$ being a universal constant. 
Using the plain bound 
$\E |\widetilde Y|^3\le( \log N) \,\E |\widetilde Y|^2$, we therefore deduce
\begin{eqnarray*} & &\Big|\int_\R \big|\sum_{n=1}^N\frac{1}{n^{\s+i t/s_k}}\big|^{2k} \frac{\dd t}{\pi( t^2+1)}  - \big(\sum_{n=1}^N\frac{1}{n^\s
}\big)^{2k} \E \,  e^{ -  |g|}  \Big| \cr &\le  &   A\ \frac{\log N }{k^{1/2}(\E |\widetilde Y|^2)^{1/2}}\big(\sum_{n=1}^N\frac{1}{n^\s
}\big)^{2k}
\ \le \  C\ \frac{ (1-\s)\log N}{k^{1/2}}\Big(\sum_{n=1}^N\frac{1}{n^\s
}\Big)^{2k} .
 \end{eqnarray*}
And $C$ is a universal constant. By taking $c_0= \E \,   e^{ -  |g|}$, this achieves the proof.\end{proof}

%%%%%%%%%%%%%%%%%%%%%%%%%%%%%%%%%%%%%%%%%%%%%%%%%%%%%%%%%%%%%%%%%%%%%%%%%%%%%% 
 \section{Concluding Remarks.}
  %%%%%%%%%%%%%%%%%%%%%%%%%%%%%%%%%%%%%% %\subsection{Connection with spectral theory of $\boldsymbol{ K(x,y)}$}
The questions treated in \cite{Wi} are also considered in \cite{Wi1} in the setting of  Widom's theory of Toeplitz integral kernels and their connection with finite sections of classical inequalities, such as Carleman or Hilbert's inequality. 
\vskip 2 pt
We believe that these are really interesting and motivating questions, which  should deserve  more investigations, notably because of the connection with Dirichlet sums and the  link with other approaches.
We conclude with a simple remark concerning a second application of \eqref{Wilf} (using the Hilbert kernel $H(x,y)= (x+y)^{-1}$) given  
%Wilf has also considered 
in \cite{Wi},  where the following formula   in which $\s>1/2$ and  $\l(n)$ is the Liouville function is established,
\begin{equation}\label{l} \sum_{m=1}^\infty \frac{\l(m)}{m^s}\sum_{d|m} \frac{d^{2it}}{d+(m/d)}= \int_0^\infty \Big|\frac{\zeta \big(2(\s+\frac{1}{2})+2i(t+\theta)\big)}{\zeta \big(\s+\frac{1}{2}+i(t+\theta)\big)}\Big|^2\frac{\dd \theta}{\cosh \pi \theta}.
\end{equation}
In fact, the same arguments used to establish \eqref{l} also apply for the kernel $K(x,y)=\max(x,y)^{-1}$, and to other  arithmetical functions. More precisely, let $f(n)$ be a completely multiplicative arithmetical function. Assume that   the series
\begin{equation}\label{f}\sum_{n=1}^\infty \frac{|f(m)|}{m^{\s_0}}
\end{equation}
converges for some  $\s_0>1$. Let $F(z)=\sum_{m=1}^\infty \frac{f(m)}{m^z}$.  Then for $\s\ge \s_0-\frac{1}{2}$, (recalling that $s=\s +it$),
\begin{equation}\label{f1} \sum_{m=1}^\infty \frac{f(m)}{m^s}\Big(\sum_{d|m} \frac{d^{2it}}{\max\big(d, (m/d)\big)}\Big)= \frac{1}{2\pi}\int_\R  \big|F\big(\s + \frac{1}{2}+ i(t+\theta)\big) \big|^2 \frac{\dd \theta }{\frac{1}{4}+  \theta^2}.
\end{equation}
 
Indeed,   by \eqref{Wilf}, 
\begin{equation*} \sum_{n,m=1}^N\frac{f(m)f(n)}{m^{\bar s}n^s} K(m,n)=\frac{1}{2\pi}\int_\R   \Big|\sum_{n=1}^N 
\frac{f(n)}{ n^{\s + \frac{1}{2}+ i(t+\theta)}}   \Big|^2 \frac{\dd \theta }{\frac{1}{4}+  \theta^2}.
\end{equation*}
Now
\begin{equation}\label{f2}\sum_{n,m=1}^N\frac{f(m)f(n)}{m^{\bar s}n^s} K(m,n)=\sum_{\nu=1}^{N^2}\frac{f(\nu)}{\nu^s}\Big( \sum_{d|\nu\atop \frac{\nu}{N}\le d\le N} K(d,\frac{\nu}{d})d^{2it}\Big)
\end{equation}
Further 
\begin{equation*}\Big| \sum_{d|\nu\atop \frac{\nu}{N}\le d\le N} K(d,\frac{\nu}{d})d^{2it}\Big|\le \sum_{d|\nu\atop \frac{\nu}{N}\le d\le N} \frac{1}{\max( d, \nu/d)} =  \frac{1}{\sqrt \nu}\sum_{d|\nu\atop \frac{\nu}{N}\le d\le N} \frac{1}{\max( d/\sqrt \nu, \sqrt \nu/d)} \le  \frac{d(\nu)}{\sqrt \nu},\end{equation*}
where $d(n)$ is the divisor function (counting the number of divisors of the natural $n$), and we recall that $d(n)= \mathcal O_\e (n^\e)$. Hence by assumption \eqref{f}, \eqref{f2} and letting $N$ tend to infinity, the result follows.

%%%%%%%%%%%%%%%%%%%%%%%%%%%%%%%%%%%%%%%%%%%%%%%%%%%%%%%%%%%%%%%%%%%%%%%%%%%%%%

 \end{document}